\documentclass[12pt,a4paper]{amsart}
\usepackage{graphics}
\usepackage{epsfig}
\usepackage{graphicx}
\theoremstyle{plain}
\usepackage{amssymb}
\usepackage[ukrainian,english]{babel}
\advance\hoffset-20mm \advance\textwidth40mm

\newtheorem{theorem}{Theorem}
\newtheorem{lemma}{Lemma}
\newtheorem*{theo*}{Theorem}

\newtheorem{corollary}{Corollary}
\theoremstyle{definition}

\newtheorem*{definition*}{Definition}

\begin{document}
\sloppy
\title[Wildness of the problem of classifying  nilpotent Lie algebras...]
{Wildness of the problem of classifying  nilpotent \\  Lie algebras of vector fields in four variables}
\author
{V. M. Bondarenko, A. P.  Petravchuk}
\address{V.M.Bondarenko:
Institute of Mathematics, National Academy of Sciences of Ukraine,
Tereschenkivska street, 3, 01004 Kyiv, Ukraine}
\email{vitalij.bond@gmail.com}
\address{Anatoliy P. Petravchuk:
Department of Algebra and Mathematical Logic, Faculty of Mechanics and Mathematics,
Taras Shevchenko National University of Kyiv, 64, Volodymyrska street, 01033  Kyiv, Ukraine}
\email{aptr@univ.kiev.ua , apetrav@gmail.com}
\date{\today}
\keywords{nilpotent Lie algebra, wild problem, vector field, derivation}
\subjclass[2000]{Primary 17B66, 16G60; Secondary 34C14}

%
\begin{abstract}
Let  $\mathbb F$ be a field $\mathbb F $ of characteristic zero. Let $W_{n}(\mathbb F)$ be the Lie algebra of all
$\mathbb F$-derivations with  the Lie bracket $[D_1, D_2]:=D_1D_2-D_2D_1$ on the polynomial ring $\mathbb F [x_1, \ldots , x_n]$. The problem of classifying  finite dimensional subalgebras of $W_{n}(\mathbb F)$ was solved  if $ n\leq 2$ and $\mathbb F=\mathbb C$ or $\mathbb F=\mathbb R.$ We prove that this problem is wild if $n\geq 4$, which means that it contains the classical unsolved problem of classifying matrix pairs up to similarity. The structure of finite dimensional subalgebras of  $W_{n}(\mathbb F)$ is interesting since each derivation in case $\mathbb F=\mathbb R$ can be considered as a vector field with polynomial coefficients on the manifold $\mathbb R^{n}.$
 \end{abstract}
\maketitle
\centerline{Dedicated to the 70-th birthday  of Professor V.V.~Sergeichuk}

\section{Introduction}

Let $\mathbb F $ be a field of characteristic zero. An \emph{$\mathbb F$-derivation} of the polynomial ring $\mathbb F [x_1, \ldots , x_n]$ is an $\mathbb F$-linear map $D: \mathbb F [x_1, \ldots , x_n]\to \mathbb F [x_1, \ldots , x_n]$ satisfying the Leibniz's rule
$D(fg)=D(f)g+fD(g)$. Every $\mathbb F$-derivation of $\mathbb F [x_1, \ldots , x_n]$ can be uniquely written in the form
$$f_1\frac{\partial}{\partial x_1}+\cdots +f_n\frac{\partial}{\partial x_n},$$ in which $f_1, \ldots , f_n\in \mathbb F [x_1, \ldots , x_n]$ and $\frac{\partial}{\partial x_i}$  are the partial derivatives.
We denote  by $W_{n}(\mathbb F)$ the vector space over $\mathbb F$ of all $\mathbb F$-derivations of $\mathbb F [x_1, \ldots , x_n]$. Clearly, $W_{n}(\mathbb F)$ is the Lie algebra with Lie bracket
$$[D_1, D_2]:=D_1D_2-D_2D_1,\qquad D_1, D_2\in W_{n}(\mathbb F).$$

The structure of finite dimensional subalgebras of the Lie algebra $W_{n}(\mathbb R)$ is interesting because its derivations can be considered as vector fields on the manifold $\mathbb R^{n}$ with polynomial coefficients.
Finite dimensional Lie algebras of vector fields on smooth manifolds correspond to Lie groups of local diffeomorphisms, which are closely connected with symmetries of PDE's; see  \cite{Lie1}.

Finite dimensional subalgebras of $W_1(\mathbb R),$ $W_1(\mathbb C)$, and $W_2(\mathbb C)$ were classified by Sophus Lie \cite{Lie1}. The classification of finite dimensional subalgebras of  $W_2(\mathbb R)$ was completed by Gonz\'alez-L\'opez, Kamran, and Olver \cite{Olver2}. Note that the papers \cite{Lie1,Olver2} describe the wider class of Lie algebras of vector fields with analytical coefficients.

The problem of classifying finite dimensional subalgebras of $W_3(\mathbb R)$ and $W_3(\mathbb C)$ remains open. They were partially descried by Sophus Lie \cite{Lie1}.  Their classification by Amaldi \cite{Amaldi1,Amaldi2} is uncomplete.

We prove that the problem of classifying finite dimensional subalgebras of $W_4(\mathbb F)$ is wild, and so it is hopeless. A classification problem is \emph{wild} if it contains the problem of classifying pairs of square matrices of the same size up to similarity transformations
\[
(A,B)\mapsto (S^{-1}AS,S^{-1}BS),\qquad S\text{ is nonsingular;}
\]
the other problems are \emph{tame}. These
terms were introduced by Donovan and Freislich \cite{Donovan} in analogy with the partition of
animals into tame and wild ones.
Precise definitions and the proof that each classification problem is either wild or tame are given by Drozd \cite{Drozd,Drozd1}
for a wide class of classification problems that includes the problem of classifying representations of finite-dimensional algebras (a geometric form of Drozd's theorem is given in \cite{gab_ukr,serg_can}).

The ground for these notions was given by
Gelfand and Ponomarev \cite{GP} in 1969; they have proved that the problem of classifying pairs of
commuting nilpotent matrices up to similarity contains the problem of classifying matrix
$t$-tuples up to similarity for each $t$. It also contains the problem of classifying representations of arbitrary quivers (i.e., arbitrary systems of vector spaces and linear mappings) and partially ordered sets; see \cite{bel_compl}.
Thus, all wild problems have the same complexity and the solution of one of them would imply the solution of each other. Analogous statements for systems of linear
and semilinear mappings and for systems of tensors of order at most 3 are given in \cite{de_semi} and \cite{fut_tens}. The wildness of the problem of classifying some classes of algebras and Lie algebras is proved in \cite{BBLPS,FKPS}.

Our main result is Theorem 2, which states that \emph{the problem of classifying finite dimensional nilpotent Lie algebras of vector fields in four variables is wild.} The case $n=3$ is discussed by Doubrov  \cite{Doubrov}.
\section{Preliminaries}

All Lie algebras that we consider are over a field $\mathbb F$ of characteristic zero. The \emph{triangular algebra} $u_{n}(\mathbb F)$ is the subalgebra  of the Lie algebra $W_n(\mathbb F)$ consisting of all derivations of the form
$$D=f_1(x_2, \ldots , x_n)\frac{\partial}{\partial x_1}+f_2(x_3, \ldots , x_n)\frac{\partial}{\partial x_2}+\cdots  +f_n\frac{\partial}{\partial x_n},$$ in which $f_n\in \mathbb F$. This Lie algebra is studied by Bavula \cite{Bavula}; it is locally nilpotent but not nilpotent (recall that a Lie algebra is \emph{locally nilpotent} if every its finitely generated subalgebra is nilpotent).

Each pair of commuting matrices $S, T\in M_n(\mathbb F)$ can be associated with the module $V_{S, T}$ over $\mathbb F[x, y]$ as follows. Let $V$ be any $n$-dimensional vector space over $\mathbb F$ with a fixed basis. The matrices $S$ and $T$ define commuting linear operators on $V$, which we denote by the same letters. Then  $V$ with multiplication $$
f(x, y)\cdot v:=f(S, T)(v),\qquad f(x, y)\in \mathbb F[x, y],\ v\in V
$$
is the left module  over $\mathbb F[x, y]$, which we denote by $V_{S, T}$.

We call $V_{S, T}$ \emph{nilpotent} if the linear operators $S$ and $T$ are nilpotent. The problem of classifying all $\mathbb F[x, y]$-modules is wild because it is equivalent to the problem of classifying matrix pairs up to similarity.

Let  $L_1$ and $L_2$ be Lie algebras over $\mathbb F $. Let $\varphi : L_1 \to \operatorname{Der}{L_2}$ be a homomorphism of Lie algebras, in which $\operatorname{Der}{L_2}$ is the Lie algebra of all $\mathbb F$-derivations of $L_2.$ The \emph{external semidirect product} of $L_1$ and $L_2$ (denoted by $L_1\rightthreetimes _{\varphi}L_2$ or $L_1\rightthreetimes L_2$) is the vector space $L_1\oplus L_2$ with the  Lie bracket $$[(a_1, b_1) , (a_2, b_2)]=([a_1, a_2], D_1(b_2)-D_2(b_1)),$$ in which $D_1:=\varphi (a_1)$ and $D_2:=\varphi (a_2)$.
If a Lie algebra $L$ contains both an ideal $N$ and a subalgebra $B$ such that $L=N+B$ and $ N\cap B=0,$  then $L$ is the \emph{internal semidirect product}  of  $B$ and $N$ (we write $L=B\rightthreetimes N$) with the natural homomorphism of $B$ into the Lie algebra $\operatorname{Der}_{\mathbb F}N.$

\section{The universal $\mathbb F[x, y]$-module}

Let us construct  a module  of countable dimension that contains an isomorphic copy of each
finite-dimensional nilpotent module over $\mathbb F[x, y].$  Denote by $W_{xy,z}$ the underlying vector space of the polynomial
algebra ${\mathbb F}[x,y,z]$ and consider the linear operators $\frac{\partial}{\partial x}$ and $\frac{\partial}{\partial y}$ on it. Since $[\frac{\partial}{\partial x}, \frac{\partial}{\partial y}]=0,$ the  vector space
$W_{xy,z}$  is   a left ${\mathbb F}[x,y]$-module with multiplication
\[
x\cdot f:= \frac{\partial f}{\partial x},\quad y\cdot f:= \frac{\partial f}{\partial y}\qquad\text{for all } f\in {\mathbb F}[x,y,z].
\]  Note that $\operatorname{Ker} \frac{\partial}{\partial x}\cap \operatorname{Ker} \frac{\partial}{\partial y}=K[z]$.

We think that the following statement is known, but we do not know where it is published.

\begin{lemma}\label{lem1}
Let $f,g\in {\mathbb F}[x,y,z]$ satisfy $\frac{\partial f}{\partial y}=\frac{\partial g}{\partial x}$. Then there exists $h\in {\mathbb F}[x,y,z]$ such that
$\frac{\partial h}{\partial x}=f$ and $\frac{\partial h}{\partial y}=g$.
The polynomial $h$  is  determined by $f$ and $g$ uniquely up to summands from ${\mathbb F}[z]$.
\end{lemma}

\begin{proof}
If such a polynomial exists, then we have $h'_{x}=f$ and so $h(x, y, z)=\int f(x, y, z)dx+u(y, z),$ in which $\int fdx$ is a primitive polynomial function  and $u(y, z)$ is an arbitrary polynomial in variables $y$ and $z.$ Differentiating both the sides of this equality, we find
$$h'_{y}=g(x, y, z)=\frac{\partial}{\partial y}\int fdx+u'_{y}(y, z).$$
Thus, $u'_{y}=g(x,y,z)-\frac{\partial}{\partial y}\int fdx.$
Integrating on $y$ this equality,  we obtain
$$u(y,z)=\int g(x,y,z)dy+\int\Big(\frac{\partial}{\partial y}\int fdx\Big)dy+v(z),$$
in which $v(z)$ is an arbitrary polynomial over $\mathbb F.$
Hence, the polynomial $h(x,y,z)=\int f(x,y,z)dx+u(y,z)$ satisfies $h'_{x}=f$ and $h'_{y}=g.$ If there exists a polynomial $h_1$ satisfying $(h_{1}')_{x}=f$ and $(h_{1}')_{y}=g,$ then $(h-h_{1})'_{x}=0$ and $(h-h_{1})'_{y}=0$;
therefore, $h-h_1\in \mathbb F[z].$
\end{proof}

\begin{theorem}\label{th-iso}
Let $V$ be a finite-dimensional vector space over a field ${\mathbb F}$. Let $S$ and $T$ be
commuting nilpotent linear operators on $V$. Then the ${\mathbb F}[x,y]$-module $V_{S,T}$  is isomorphic to a
 submodule  of the $\mathbb F[x,y]$-module $W_{xy,z}$.
\end{theorem}
\begin{proof}
We use induction on $n:=\dim_{\mathbb F} V$.
If $n=1,$ then $W_{xy,z}=\{a1|a\in{\mathbb F}\}$; we take any $0\ne v\in V$ and define the monomorphism $\varphi:V_{S,T}\to W_{xy,z}$ by $\varphi(v)=1$.

Let $n>1.$ Since the operators $S$ and $T$ are nilpotent, there exists an $(n-1)$-dimensional subspace $U\subset V_{S,T}$
that is invariant for the linear operators $S$ and $T$ (i.e., $U$ is a submodule of $V_{S,T}$). By the inductive hypothesis,  there exists  an isomorphism $\varphi: U\to N$  on some submodule $N$ of $W_{xy,z}$.
Take an arbitrary  $v\in V_{S,T}\setminus U$.  Since  $S$ and $T$ are nilpotent, we have   $S(v)\in U$ and $T(v)\in U$, which means that the images $\varphi(S(v))$ and  $\varphi(T(v))$ belong to $N$, and therefore they   are polynomials of $x,y,z$. Write $\varphi(S(v))=f(x,y,z)$ and  $\varphi(T(v))=g(x,y,z)$.
The equalities
\begin{equation}\label{hjk}
\begin{split}
\varphi(T(S(v)))&=\frac{\partial}{\partial y}(\varphi(S(v))=\frac{\partial f}{\partial y}
\\
\varphi(S(T(v)))&=\frac{\partial}{\partial x}(\varphi(T(v))=\frac{\partial g}{\partial x}
\end{split}
\end{equation}
hold since $S(v)\in U$ and $T(v)\in U$, and because the actions of $S$ and $T$ on $U$ induce  the actions of $\frac{\partial }{\partial x}$ and $\frac{\partial}{\partial y}$ on $N$. By the conditions of the theorem,
$ST=TS$, and so \eqref{hjk} ensure $\frac{\partial f}{\partial y}=\frac{\partial g}{\partial x}$. By Lemma \ref{lem1}, there exists a polynomial $h=h(x,y,z)\in {\mathbb F}[x,y,z]$ such that
\begin{equation}\label{asq}
\frac{\partial h}{\partial x}=f,\qquad \frac{\partial h}{\partial y}=g.
\end{equation}

We can suppose that $h\not \in N.$ Indeed,
if $h\in N$, then we take $u(z)\in \mathbb F[z]$ of sufficiently high degree,  replace $h$ by the polynomial $h_1=h+u(z)$, and get  $h_1\notin N$ for which \eqref{asq} holds with $h_1$ instead of $h$.

The vector subspace $N+{\mathbb F}\langle h \rangle$ is a submodule of the
${\mathbb F}[x,y]$-module $W_{xy,z}$.
Let
$$
\varphi_1: V=U+{\mathbb F}\langle v \rangle\to N+{\mathbb F}\langle h \rangle
$$
be the  ${\mathbb F}$-linear map
such that  $\varphi_1(u)=\varphi(u)$ on $U$ and $\varphi_1(v)=h$.
It is verified by  direct calculations that  $\varphi_1$ is an isomorphism of  the ${\mathbb F}[x]$-module
$V_{S,T}$ on the submodule   $N+{\mathbb F}\langle h \rangle$ of $W_{xy,z}.$
 \end{proof}

Let $\theta$ be a linear automorphism of $\mathbb{F}[x,y]$ determined by some linear polynomials  \[
\theta(x)=\alpha_{11}x+\alpha_{12}y,
\qquad \theta(y)=\alpha_{21}x+\alpha_{22}y,
\]
in which $\alpha_{ij}\in \mathbb{F}$ and $\det[\alpha_{ij}]\ne 0$. For any  $\mathbb{F}[x,y]$-module $V$, the automorphism $\theta$  defines the  ``twisted''  module $V_\theta$ by the rule
\[
x\circ v=\theta(x)\cdot v,\quad   y\circ v=\theta(y)\cdot v\qquad \text{for all }v\in V,
\]
where the multiplication in the right-hand sides is the multiplication in the module $V$.

We say that $\mathbb{F}[x,y]$-modules $V$ and $W$ are \emph{weakly isomorphic} if there exists an $\mathbb{F}[x,y]$-module $U$ such that $V$ and $U$ are isomorphic and $W=U_\theta$ for some linear automorphism $\theta$.
By \cite{FKPS}, the problem of classifying  finite-dimensional nilpotent $\mathbb{F}[x,y]$-modules up to weak isomorphism is wild.
Thus, Theorem \ref{th-iso} ensures the following corollary.

\begin{corollary}\label{corol-1}
The problem of classifying finite-dimensional submodules of
$W_{xy,z}$ up to weak isomorphism is wild.
\end{corollary}

\section{The main theorem}

\begin{lemma}\label{wild}
The problem of classifying finite dimensional  nilpotent Lie algebras of the form $L=B\rightthreetimes W,$ in which $W$ is an abelian ideal of $L$ of codimension $2,$ is wild.
\end{lemma}

\begin{proof}
By \cite{PetrSysak}, the problem of classifying nilpotent finite dimensional Lie algebras of the form $L=B\rightthreetimes W,$ in which $W$ is an abelian ideal of $L$ of codimension $2,$ is equivalent to the problem of classifying, up to weak isomorphism, of nilpotent finite dimensional $\mathbb F[x, y]$-modules.  The latter problem is wild by \cite[Theorem 1]{FKPS}.
\end{proof}

\begin{theorem}
The problem of classifying finite dimensional nilpotent subalgebras of the Lie algebra $W_4(\mathbb F)$ is wild.
\end{theorem}
\begin{proof}
Let us consider the Lie subalgebra $M$ of $W_4(\mathbb F)$ consisting of all derivations on $\mathbb F[x, y, z, t]$ of the form
$$D=\alpha \frac{\partial}{\partial x}+\beta \frac{\partial}{\partial y}+f(x, y, z) \frac{\partial}{\partial t},$$
in which $\alpha , \beta\in \mathbb F$ and $f(x, y, z)\in \mathbb F[x, y, z].$  Clearly, $M$ is contained in  the triangular subalgebra $u_4(\mathbb F).$ Let $I$ be the abelian ideal of $M$ that consists of all the derivations of the form $f(x, y, z) \frac{\partial}{\partial t}$. Then $I$ has codimension $2$ in $M$ and it can be considered as an $\mathbb F[x, y]$-module with respect to the multiplication
\begin{align*}
&x\cdot f(x, y, z) \frac{\partial}{\partial t}= \frac{\partial f}{\partial x}\frac{\partial}{\partial t}
\ &&(\text{which is } \left[\frac{\partial}{\partial x}, f(x, y, z) \frac{\partial}{\partial t}\right] \  \mbox{in }  M),
                \\
&y\cdot f(x, y, z) \frac{\partial}{\partial t}= \frac{\partial f}{\partial y}\frac{\partial}{\partial t}
&&(\text{which is } \left[\frac{\partial}{\partial y}, f(x, y, z) \frac{\partial}{\partial t}\right] \  \mbox{in }  M).
\end{align*}
This module  $I$ over  $\mathbb F[x, y]$ is isomorphic to $W_{xy, z}$.

Denote by $\mathcal N$ the set of all finite dimensional subalgebras of the Lie algebra $M$  of the form
 $$L=\mathbb F\left\langle  \frac{\partial}{\partial x},  \frac{\partial}{\partial y}\right\rangle +V,$$
in which $V$ is a finite dimensional ideal of $M$ that is contained in $I$ (in other words, $L$ is a finite dimensional subalgebra of $M$ containing the elements $\frac{\partial}{\partial x}$ and $\frac{\partial}{\partial y}$).

Let us show that the problem of classifying Lie algebras from  $\mathcal N$ is wild. Take any finite dimensional nilpotent Lie algebra $L$ of the form $L=B\rightthreetimes W,$ where $W$ is an abelian ideal of $L$ of codimension $2$. Choosing a basis $S, T$ of $B$ over $\mathbb F,$  we associate  with $L$ the finite dimensional  $\mathbb F[x, y]$-module $W_{S, T}$.  The $\mathbb F[x, y]$-modules $I$ and $W_{xy, z}$ are isomorphic; by Theorem 1 there exists a submodule $U$ of $I$ that is isomorphic to  $W_{S, T}$. By \cite[Proposition 1]{PetrSysak},
the subalgebra $\mathbb F\langle \frac{\partial}{\partial x}, \frac{\partial}{\partial y}\rangle \rightthreetimes U$  of the Lie algebra $M$ is isomorphic to $L$.  By Lemma \ref{wild}, the problem of classifying  subalgebras from $\mathcal N$ is wild.
\end{proof}

\begin{corollary}
Let  ${\overline W}_{n}(\mathbb R)$ with $n\geq 4$  be the Lie algebra of all vector fields on $\mathbb R^{4}$ with analytical coefficients. Then the problem of classifying finite dimensional nilpotent subalgebras of ${\overline W}_{n}(\mathbb R)$  is wild.
\end{corollary}


%

\begin{thebibliography}{99}
\bibitem{Amaldi1} U. Amaldi,  Contributo all determinazione dei gruppi continui finiti dello spazio ordinario I, Giornale Mat. Battaglini Prog. Studi Univ. Ital. 39  (1901) 273--316.

\bibitem{Amaldi2} U. Amaldi,  Contributo all determinazione dei gruppi continui finiti dello spazio ordinario II, Giornale Mat. Battaglini Prog. Studi Univ. Ital. 40  (1902) 105--141.

 \bibitem{Bavula} V.V. Bavula,
Lie algebras of triangular polynomial derivations and an isomorphism criterion for their Lie factor algebras, Izv. Math. 77 (2013) 1067--1104.


\bibitem{BBLPS}
G. Belitskii, V.M. Bondarenko, R. Lipyanski, V.V. Plachotnik, V.V.~Sergeichuk, The problems of
classifying pairs of forms and local algebras with zero cube radical are wild, Linear Algebra Appl.
402 (2005) 135--142.

\bibitem{bel_compl}
G.R. Belitskii, V.V. Sergeichuk, Complexity of matrix problems, Linear Algebra Appl. 361 (2003) 203--222.


\bibitem{Donovan}
P. Donovan, M.R. Freislich, Some evidence for an extension of the Brauer--Thrall conjecture, Sonderforschungsbereich
Theor. Math. 40 (1972) 24--26.


 \bibitem{Doubrov} B. Doubrov, Three-dimensional homogeneoous spaces with non-solvable transformation groups, 2017, available from https://arxiv.org/abs/1704.04393


\bibitem{Drozd} Ju.A. Drozd, Tame and wild matrix problems, in: Representations and Quadratic Forms, pp. 39--74, Akad. Nauk Ukrain. SSR, Inst. Mat., Kiev, 1979 (in Russian; English translation: Amer. Math. Soc. Transl. 128 (1986) 31--55).

\bibitem{Drozd1}
Ju.A. Drozd, Tame and wild matrix problems, in: Representation theory, II (Proc. Second Internat. Conf., Carleton Univ., Ottawa, Ont., 1979), pp. 242--258, Lecture Notes in Math., 832, Springer, Berlin-New York, 1980.

\bibitem{fut_tens}
V. Futorny, J.A. Grochow, V.V. Sergeichuk, Wildness for tensors, submitted to Linear Algebra Appl.

\bibitem{FKPS} V. Futorny, T. Klymchuk, A.P. Petravchuk, V.V. Sergeichuk, Wildness of the problems of classifying two-dimensional spaces of commuting linear operators and certain Lie algebras, Linear Algebra Appl. 536 (2018) 201--209.

\bibitem{gab_ukr}
P. Gabriel, L.A. Nazarova, A.V. Roiter, V.V. Sergeichuk, D.~Vossieck, Tame and wild subspace problems, Ukr. Math. J. 45 (1993) 335--372.

 \bibitem{GP}
I.M.~Gelfand, V.A.~Ponomarev,  Remarks on the classification of a pair of commuting linear transformations in a finite dimensional space, Functional Anal. Appl. 3 (1969) 325--326.

\bibitem{Olver2} A. Gonz\'alez-L\'opez, N. Kamran, P.J. Olver,  Lie algebras of vector fields in the real plane, Proc. London Math. Soc. (3) 64 (no. 2) (1992) 339--368.

\bibitem{de_semi}
D.D. de Oliveira, R.A. Horn, T. Klimchuk, V.V. Sergeichuk, Remarks
on the classification of a pair of commuting semilinear operators, Linear
Algebra Appl. 436 (2012) 3362--3372.


\bibitem{Lie1}  S. Lie, Theorie der Transformationsgruppen, Bd. 3, Teubner, Leipzig, 1893.


\bibitem{PetrSysak} A.P. Petravchuk, K.Ya. Sysak, Lie algebras associated with modules over polynomial rings, Ukrainian Math. J. 69 (2018) 1433--1444.

\bibitem{serg_can}
V.V. Sergeichuk, Canonical matrices for linear matrix problems, Linear Algebra Appl. 317 (2000) 53--102.


\end{thebibliography}
\end{document}